\documentclass[a4paper,12pt,reqno]{amsart}

\usepackage[utf8]{inputenc}
\usepackage[T1]{fontenc}
\usepackage[english]{babel}

\usepackage{amsmath,amssymb,amsfonts,amsthm}
\usepackage{mathtools}
\usepackage{mathrsfs}

\usepackage[%
	left=2.5cm,       
	right=2.5cm,      
	top=3.5cm,        
	bottom=3.5cm,     
	heightrounded,    
	bindingoffset=0mm 
]{geometry}

\usepackage{hyperref} 
\usepackage[noabbrev,capitalize]{cleveref}
\usepackage{bookmark}

\usepackage{enumerate}
\usepackage{xcolor}

\numberwithin{equation}{section}

\theoremstyle{plain}
	\newtheorem{theorem}{Theorem}[section]
	\newtheorem{proposition}[theorem]{Proposition}
	\newtheorem{lemma}[theorem]{Lemma}
	\newtheorem{corollary}[theorem]{Corollary}

\theoremstyle{definition}

	\newtheorem{remark}[theorem]{Remark}

\newcommand{\N}{\mathbb{N}}
\newcommand{\R}{\mathbb{R}}

\newcommand{\eps}{\varepsilon}

\let\altphi\phi
\let\phi\varphi
\let\varphi\altphi
\let\altphi\undefined

\DeclarePairedDelimiter{\scalar}{<}{>}                                     
\DeclarePairedDelimiter{\set}{\{}{\}}
\DeclarePairedDelimiter{\abs}{|}{|}

\newcommand{\weakstarto}{\overset*\rightharpoonup}

\let\div\undefined
\DeclareMathOperator{\div}{div}

\newcommand{\M}{\mathsf{M}}
\newcommand{\wM}{\closure[-0.5]{\M}}
\newcommand{\A}{\mathsf{A}}

\DeclareMathOperator{\Lip}{Lip}

\newcommand{\leb}{\mathscr{L}}

\def\Xint#1{\mathchoice
	{\XXint\displaystyle\textstyle{#1}}%
	{\XXint\textstyle\scriptstyle{#1}}%
	{\XXint\scriptstyle\scriptscriptstyle{#1}}%
	{\XXint\scriptscriptstyle\scriptscriptstyle{#1}}%
	\!\int}
\def\XXint#1#2#3{{\setbox0=\hbox{$#1{#2#3}{\int}$ }
		\vcenter{\hbox{$#2#3$ }}\kern-.6\wd0}}

\def\dashint{\Xint-}

\mathchardef\ordinarycolon\mathcode`\:
\mathcode`\:=\string"8000
\begingroup \catcode`\:=\active
  \gdef:{\mathrel{\mathop\ordinarycolon}}
\endgroup

\newcommand{\closure}[2][3]{%
  {}\mkern#1mu\overline{\mkern-#1mu#2}}

\usepackage[initials]{amsrefs}

\allowdisplaybreaks

\begin{document}

\title[A maximal function characterization]{A maximal function characterization of absolutely continuous measures and Sobolev functions}

\author[E. Bruè]{Elia Bruè}
\address{Scuola Normale Superiore, Piazza Cavalieri 7, 56126 Pisa, Italy}
\email{elia.brue@sns.it}

\author[Q. H. Nguyen]{Quoc-Hung Nguyen}
\address{Centro di Ricerca Matematica ``Ennio De Giorgi'', Piazza Cavalieri 3, 56126 Pisa, Italy}
\email{quochung.nguyen@sns.it}

\author[G. Stefani]{Giorgio Stefani}
\address{Scuola Normale Superiore, Piazza Cavalieri 7, 56126 Pisa, Italy}
\email{giorgio.stefani@sns.it}

\date{\today}

\keywords{Maximal functions, singular measures, $BV$ and Sobolev functions, regular Lagrangian flows}

\subjclass[2010]{42B25, 46E27, 26A45}

\thanks{\textit{Acknowledgements}. The authors thank Professors Luigi Ambrosio and Giovanni Alberti for useful conversations about the subject. We also thank Professor Guido De Philippis for his comments on a preliminary version of the manuscript and for having pointed us the reference~\cite{MS13}.}

\begin{abstract}
In this note, we give a new characterization of Sobolev $W^{1,1}$ functions among $BV$ functions via Hardy--Littlewood maximal function. Exploiting some ideas coming from the proof of this result, we are also able to give a new characterization of absolutely continuous measures via a weakened version of Hardy--Littlewood maximal function. Finally, we show that the approach adopted in~\cites{CDeL08,J10} to establish existence and uniqueness of regular Lagrangian flows associated to Sobolev vector fields cannot be further extended to the case of $BV$ vector fields.
\end{abstract}

\maketitle

\section{Introduction}

Let $\mu$ be a Borel measure in $\R^d$. We let
\begin{equation}\label{eq:def_max_func}
\M\mu(x):=\sup_{r>0}\frac{|\mu|(B(x,r))}{\leb^d(B(x,r))}\in[0,+\infty]
\end{equation} 
be the \emph{(Hardy--Littlewood) maximal function of $\mu$} at $x\in\R^d$, see~\cite{S70}*{Chapter 1}. If the measure $\mu$ is absolutely continuous with respect to the Lebesgue measure with density $f\in L^1_{\rm loc}(\R^d)$, then we can rewrite~\eqref{eq:def_max_func} as
\begin{equation}\label{eq:def_max_func_f}
\M f(x):=\sup_{r>0}\dashint_{B(x,r)} |f(y)|\ dy\in[0,+\infty]
\end{equation} 
for all $x\in\R^d$. It is well known that if $f\in L^p(\R^d)$ for some $1<p\le+\infty$, then the maximal function in~\eqref{eq:def_max_func_f} satisfies the following \emph{strong $(p,p)$-type estimate}
\begin{equation}\label{eq:(p,p)_estimate}
\|\M f\|_{L^p(\R^d)}\lesssim_d\|f\|_{L^p(\R^d)}.
\end{equation}
Here and in the following, given two quantities $A$ and $B$, we write $A\lesssim_d B$ (resp.~$A\gtrsim_d B$) if there exists a dimensional constant $C>0$ such that $A\le C B$ (resp.~$A\ge C B$).
If $f\in L^1(\R^d)$, then the maximal function in~\eqref{eq:def_max_func_f} satisfies the following \emph{weak $(1,1)$-type estimate}
\begin{equation}\label{eq:(1,1)_estimate}
\sup_{\lambda>0} \lambda\,\leb^d(\set*{x\in\R^d: \M f(x)>\lambda})\lesssim_d\|f\|_{L^1(\R^d)}.
\end{equation}
For a proof of inequalities~\eqref{eq:(p,p)_estimate} and~\eqref{eq:(1,1)_estimate}, we refer the interested reader to~\cite{S70}*{Theorem~1}. Actually, if $f\in L^1(\R^d)$, writing $f=f\chi_{\set*{|f|>\frac{\lambda}{2}}}+f\chi_{\set*{|f|\le\frac{\lambda}{2}}}$ and combining~\eqref{eq:(p,p)_estimate} and~\eqref{eq:(1,1)_estimate}, we can improve~\eqref{eq:(1,1)_estimate} as
\begin{equation*}
\lambda\,\leb^d(\set*{x\in\R^d: \M f(x)>\lambda})\lesssim_d\int_{|f|>\frac{\lambda}{2}}|f(y)|\ dy,
\end{equation*}
so that
\begin{equation*}
\limsup_{\lambda\to+\infty}\lambda\,\leb^d(\set*{x\in\R^d: \M f(x)>\lambda})=0.
\end{equation*}
With a similar reasoning, we also get that
\begin{equation}\label{eq:standard_limsup}
\limsup_{\lambda\to+\infty} \lambda\,\leb^d(\set{x\in\R^d : \M\mu(x)>\lambda}\lesssim_d |\mu^s|(\R^d),
\end{equation}
for all finite Borel measures~$\mu$ (see~\cite{N18}*{Section~2} for more details). Here and in the following, $\mu^s$ denotes the \emph{singular part} of the measure~$\mu$ with respect to the Lebesgue measure in~$\R^d$. 

As remarked in~\cite{MS13}*{Problem 3.5}, it is also possible to establish a reverse version of inequality~\eqref{eq:standard_limsup}, see \cref{prop:main} below.

\begin{proposition}\label{prop:main}
Let $\mu$ be a finite Borel measure in~$\R^d$. Then
\begin{equation}\label{eq:main}
\inf_{\lambda>0} \lambda\,\leb^d(\set*{x\in\R^d: \M\mu(x)>\lambda})\gtrsim_d|\mu^s|(\R^d).
\end{equation}
More in general, given a cube $Q\subset\R^d$, it holds
\begin{equation}\label{eq:main_local}
\inf\set*{\lambda\,\leb^d(\set{x\in Q: \M\mu(x)>\lambda}) : \lambda>\tfrac{|\mu|(Q)}{\leb^d(Q)}}\gtrsim_d |\mu^s|(Q).
\end{equation} 
\end{proposition}

For the reader's convenience, we give a proof of \cref{prop:main} in \cref{sec:proof_prop_main}.

Combining inequalities~\eqref{eq:standard_limsup} and~\eqref{eq:main}, we immediately deduce the following characterization of absolutely continuous measures in~$\R^d$.

\begin{corollary}\label{coroll:characterization}
Let $\mu$ be a finite Borel measure in $\R^d$. Then $\mu\ll\leb^d$ if and only if 
\begin{equation}\label{eq:coroll_characterization}
\lim_{\lambda\to+\infty} \lambda\,\leb^d(\set*{x\in\R^d: \M\mu(x)>\lambda})=0.
\end{equation}
\end{corollary}

Inspired by \cref{coroll:characterization}, we give a new characterization of Sobolev~$W^{1,1}$ among $BV$ functions in term of the behaviour of a suitable maximal function. For $f\in L^1_{\rm loc}(\R^d)$, we define
\begin{equation}\label{eq:def_A}
\A f(x):=\sup_{r>0}\frac{1}{r}\dashint_{B(x,r)} |f-(f)_{x,r}|\ dy\in[0,+\infty]
\end{equation}
for all $x\in\R^d$, where $(f)_{x,r}=\dashint_{B(x,r)} f\,dy$. Note that, by Poincaré's inequality and by inequality~\eqref{eq:standard_limsup} applied to the measure $\mu=D f$, we have that
\begin{equation*}
\limsup_{\lambda\to+\infty}\lambda\,\leb^d\big(\set*{x\in\R^d : \A f(x)>\lambda}\big)\lesssim_d|D^s f|(\R^d).
\end{equation*}
Our main result is the following, see \cref{sec:BV_vs_S} for the proof.

\begin{theorem}\label{th:BV_vs_S}
	Let $f\in BV(\R^d)$. Then 
	\begin{equation*}
	\liminf_{\lambda\to+\infty}\lambda\,\leb^d\big(\set*{x\in\R^d : \A f(x)>\lambda}\big)\gtrsim_d|D^s f|(\R^d).
	\end{equation*}
	In particular, $f\in W^{1,1}(\R^d)$ if and only if
	\begin{equation*}
	\liminf_{\lambda\to+\infty}\lambda\,\leb^d\big(\set*{x\in\R^d : \A f(x)>\lambda}\big)=0.
	\end{equation*}
\end{theorem}
Exploiting some ideas coming from the proof of the aforementioned \cref{th:BV_vs_S}, we are able to improve \cref{prop:main} as follows. For a finite Borel measure~$\mu$ on~$\R^d$ (possibly with sign), we define
\begin{equation}\label{eq:def_weak_max_func}
\wM\mu(x)
:=\sup_{r>0}\frac{|\mu(B(x,r))|}{\leb^d(B(x,r))}\in[0,+\infty]
\end{equation}
for all $x\in\R^d$. Note that the maximal function defined in~\eqref{eq:def_weak_max_func} is weaker than the one recalled in~\eqref{eq:def_max_func}, in the sense that $\wM\mu(x)\le\M\mu(x)$ for all $x\in\R^d$. Then the following result holds, see \cref{sec:proof_main} for the proof.

\begin{theorem}\label{th:weak_main}
Let~$\mu$ be a finite Borel measure on~$\R^d$ (possibly with sign). Then
\begin{equation}\label{eq:theorem_weak_main}
\liminf_{\lambda\to+\infty}\lambda\,\leb^d(\set{x\in\R^d : \wM\mu(x)>\lambda}) \gtrsim_d |\mu^s|(\R^d).
\end{equation}
\end{theorem}

Hence \cref{coroll:characterization} still holds for any (possibly signed) finite Borel measure~$\mu$ in~$\R^d$ with~\eqref{eq:coroll_characterization} replaced by
\begin{equation*}
\liminf_{\lambda\to+\infty} \lambda\,\leb^d(\set*{x\in\R^d: \wM\mu(x)>\lambda})=0.
\end{equation*}

The last goal of this note --- which was our starting motivation for the study of inequality~\eqref{eq:main} --- comes from the theory of ordinary differential equations (ODEs) with weakly differentiable vector fields. 

The study of this theory was started by DiPerna and Lions in their seminal paper~\cite{DL89}, in which they proved existence and uniqueness of solutions of ODEs with Sobolev vector fields. The extension of the results obtained in~\cite{DL89} to vector fields with $BV$~regularity was established by Ambrosio in the groundbreaking paper~\cite{A04}, where the notion of \emph{regular Lagrangian flow} was introduced as a generalization of the classical definition of flow (see~\cite{CDeL08}*{Definition~1.1} for a precise definition). More in detail, the main result of~\cite{A04} reads as follows. For a time $T\in(0,+\infty]$, let $b\colon(0,T)\times\R^d\to\R^d$ be a bounded time-dependent vector field such that 
\begin{equation}\label{eq:assumptions_B}
b\in L^1((0,T);BV(\mathbb{R}^d;\R^d)),
\quad
\div b\in L^1((0,T);L^\infty(\mathbb{R}^d;\R^d)).
\end{equation}
Consider the associated Cauchy problem
\begin{equation}\label{eq:ODE}
\begin{cases}
\dfrac{dX_t}{dt}(x)=b_t(X_t(x)), \quad (t,x)\in (0,T)\times\R^d,\\[3mm]
X_0(x)=x\in\R^d.
\end{cases}
\end{equation}
Then there exists a unique regular Lagrangian flow $X\colon[0,T)\times\R^d\to\R^d$ solving~\eqref{eq:ODE}.

Both the results by DiPerna--Lions and Ambrosio rely on the so-called \textit{Eulerian approach}, meaning that the problem~\eqref{eq:ODE} is studied indirectly via the closely linked \emph{transport equation}. A more direct approach, the so-called \emph{Lagrangian approach}, was proposed by Crippa and De Lellis in~\cite{CDeL08}, where simple \textit{a priori} estimates were exploited in order to get existence, uniqueness, compactness and even mild regularity properties of the regular Lagrangian flow associated to a Sobolev $W^{1,p}$ vector field for every~$p>1$. Their approach has been extended to the case $p=1$ by Jabin in~\cite{J10}, where it was observed that if the quantity
\begin{equation}\label{eq:decay_quantity}
\mathcal{Q}(B;\delta):=|\log\delta|^{-1}
\int_0^T\int_B\min\set*{\delta^{-1},\M|Db_t|(x)}\ dx\,dt
\end{equation}
satisfies the decay property
\begin{equation}\label{eq:decay_property}
\limsup_{\delta\to0^+}\mathcal{Q}(B;\delta)=0
\end{equation}
for all balls $B\subset\R^d$, then there exists a unique regular Lagrangian flow associated to~$b$ (for a more detailed exposition of these results, see~\cite{N18}*{Section~1}).

Using \cref{prop:main}, we can prove that~\eqref{eq:decay_property} holds true if and only if 
\begin{equation*}
b\in L^1((0,T);W_{\rm loc}^{1,1}(\R^d;\R^d)),
\end{equation*}
so that the approach by Crippa--De Lellis and Jabin cannot be further extended to the case of $BV$~vector fields. Our result reads as follows, see \cref{sec:proof_decay} for the proof.

\begin{proposition}\label{th:decay}
Let $b\colon(0,T)\times\R^d\to\R^d$ be a vector field satisfying~\eqref{eq:assumptions_B}. Then
\begin{equation}\label{eq:th_decay}
\int_0^T |D^s b_t|(B)\ dt
\lesssim_d\liminf_{\delta\to0^+}\mathcal{Q}(B;\delta)
\le\limsup_{\delta\to0^+}\mathcal{Q}(B;\delta)
\lesssim_d\int_0^T |D^s b_t|(\closure{B})\ dt
\end{equation}
for all balls $B\subset\R^d$, where $\mathcal{Q}(B;\delta)$ is as in~\eqref{eq:decay_quantity}. In particular, the decay property~\eqref{eq:decay_property} is satisfied if and only if $b\in L^1((0,T);W_{\text{loc}}^{1,1}(\R^d;\R^d))$.
\end{proposition}

\section{Proof of \cref{th:weak_main}}
\label{sec:proof_main}

In this section, we prove \cref{th:weak_main}. Write $\mu=\eta|\mu|$, where $\eta\colon\R^d\to\R$ satisfies $|\eta(x)|=1$ for $|\mu|$-a.e.\ $x\in\R^d$. For each $\eps>0$ let $\eta_\eps\in\Lip(\R^d)$ be a Lipschitz function with Lipschitz constant $C_\eps>0$ such that 
\begin{equation}\label{eq:def_eta_eps}
\int_{\R^d} |\eta(x)-\eta_\eps(x)|\ d|\mu|(x) < \eps.
\end{equation}

We claim that
\begin{equation}\label{eq:Lip_change}
\frac{1}{\omega_dr^d}\abs*{\int_{B(x,r)} \eta_\eps\ d |\mu|} 
+2C_\eps r\M|\mu|(x)
\ge
\frac{1}{\omega_dr^d}\int_{B(x,r)}|\eta_{\eps}|\ d |\mu|,
\end{equation}
for all $x\in\R^d$ and all $r>0$. Indeed, we can estimate
\begin{align*}
\abs*{\int_{B(x,r)} \eta_\eps(y)\ d |\mu|(y)}
&\geq \abs*{\int_{B(x,r)} \eta_\eps(x)\ d |\mu|(y)}
	-\int_{B(x,r)}\abs*{\eta_\eps(y)-\eta_\eps(x)}\ d |\mu|(y)\\
&\hspace*{-2cm}\geq \int_{B(x,r)} \abs*{\eta_\eps(x)}\ d|\mu|(y)
	-C_\eps r\,\int_{B(x,r)} d|\mu|(y)\\ 
&\hspace*{-2cm}\geq \int_{B(x,r)}\abs*{\eta_\eps(y)}\ d|\mu|(y)
	-\int_{B(x,r)} \abs*{\abs*{\eta_\eps(y)}-\abs*{\eta_\eps(x)}}\ d|\mu|(y)
	-C_\eps r\,\int_{B(x,r)} d|\mu|(y)\\
&\hspace*{-2cm}\geq \int_{B(x,r)}\abs*{\eta_\eps(y)}\ d|\mu|(y)
	-2C_\eps r\,\int_{B(x,r)} d|\mu|(y) 
\end{align*}
for all $r>0$, from which~\eqref{eq:Lip_change} follows.

Let us now consider the measures $\nu_{\eps}:=|\eta-\eta_{\eps}|\,|\mu|$ for $\eps>0$. We claim that
\begin{equation}\label{eq:Lip_change_eps}
\frac{1}{\omega_dr^d}\abs*{\int_{B(x,r)}\eta\ d |\mu|}
+2C_\eps r\M\mu(x)+2\M\nu_\eps(x)
\ge
\frac{1}{\omega_dr^d}\int_{B(x,r)}|\eta|\ d |\mu|
\end{equation}
for all $x\in\R^d$ and all $r>0$. Indeed, assuming $\M\mu(x)<+\infty$ and $\M\nu_\eps(x)<+\infty$ without loss of generality, by~\eqref{eq:Lip_change} we can estimate
\begin{align*}
\frac{1}{\omega_dr^d}\abs*{\int_{B(x,r)}\eta\ d |\mu|}
&\ge\frac{1}{\omega_dr^d}\abs*{\int_{B(x,r)}\eta_\eps\ d |\mu|}
	-\M\nu_\eps(x)\\
&\ge\frac{1}{\omega_dr^d}\int_{B(x,r)}|\eta_{\eps}|\ d |\mu|
	-2C_\eps r\M\mu(x)-\M\nu_\eps(x)\\
&\ge\frac{1}{\omega_dr^d}\int_{B(x,r)}|\eta|\ d |\mu|
	-2C_\eps r\M\mu(x)-2\M\nu_\eps(x).
\end{align*}

For $\tau>0$ define
\begin{equation}\label{eq:def_M_stopped}
\M^\tau\mu(x):=\sup_{0<r<\tau}\frac{|\mu|(B(x,r))}{\leb^d(B(x,r))}\in[0,+\infty],
\end{equation}
and similarly
\begin{equation*}
\wM^\tau\mu(x)
:=\sup_{0<r<\tau}\frac{|\mu(B(x,r))|}{\leb^d(B(x,r))}\in[0,+\infty],
\end{equation*}
for all $x\in\R^d$. Taking the supremum with respect to $r\in(0,\tau)$ in~\eqref{eq:Lip_change_eps}, we find
\begin{equation}\label{eq:Lip_change_tau}
\wM^\tau\mu(x)+2C_\eps\tau\M\mu(x)+2\M\nu_\eps(x)\geq\M^\tau\mu(x),
\end{equation}
for all $x\in\R^d$. 

We claim that
\begin{equation}\label{eq:liminf_tau_eps}
\liminf_{\lambda\to+\infty}\lambda\,\leb^d(\set{x\in\R^d : \wM\mu(x)>\lambda})
	+C_\eps\tau|\mu^s|(\R^d)+\eps
\gtrsim_d
|\mu^s|(\R^d),
\end{equation}
for all $\eps>0$ and $\tau>0$. Indeed, since $|\mu|(\R^d)<+\infty$, given $\tau>0$, for all $\lambda>0$ sufficiently large it holds
\begin{equation}\label{eq:magic_stop}
\set{x\in\R^d : \M^\tau\mu(x)>\lambda}
=\set{x\in\R^d : \M\mu(x)>\lambda}.
\end{equation}
Thus, on the one hand, by~\eqref{eq:main}, we have
\begin{equation*}
\begin{split}
\liminf_{\lambda\to+\infty}\lambda\,\leb^d(&\set{x\in\R^d : \M^\tau\mu(x)>\lambda})\\
&=\liminf_{\lambda\to+\infty}\lambda\,\leb^d(\set{x\in\R^d : \M\mu(x)>\lambda})
\gtrsim_d|\mu^s|(\R^d).
\end{split}
\end{equation*}
On the other hand, by~\eqref{eq:standard_limsup}, \eqref{eq:def_eta_eps} and~\eqref{eq:Lip_change_tau} we can estimate
\begin{align*}
\liminf_{\lambda\to+\infty}&\lambda\,\leb^d(\set{x\in\R^d : \M^\tau\mu(x)>\lambda})\\
&\le\liminf_{\lambda\to+\infty}\lambda\,\leb^d(\set{x\in\R^d : \wM^\tau\mu(x)+2C_\eps\tau\M\mu(x)+2\M\nu_\eps(x)>\lambda})\\
&\lesssim_d\liminf_{\lambda\to+\infty}\lambda\,\leb^d(\set{x\in\R^d : \wM^\tau\mu(x)>\lambda})\\
	&\quad+C_\eps\tau\limsup_{\lambda\to+\infty}\lambda\,\leb^d(\set{x\in\R^d : \M\mu(x)>\lambda})\\
	&\quad+\limsup_{\lambda\to+\infty}\lambda\,\leb^d(\set{x\in\R^d : \M\nu_\eps(x)>\lambda})\\
&\le\liminf_{\lambda\to+\infty}\lambda\,\leb^d(\set{x\in\R^d : \wM^\tau\mu(x)>\lambda})
	+C_\eps\tau|\mu^s|(\R^d)+\int_{\R^d}|\eta-\eta_{\eps}|\ d |\mu|\\
&\le\liminf_{\lambda\to+\infty}\lambda\,\leb^d(\set{x\in\R^d : \wM^\tau\mu(x)>\lambda})
	+C_\eps\tau|\mu^s|(\R^d)+\eps
	\\
	&=\liminf_{\lambda\to+\infty}\lambda\,\leb^d(\set{x\in\R^d : \wM\mu(x)>\lambda})
	+C_\eps\tau|\mu^s|(\R^d)+\eps,
\end{align*}
for all $\eps>0$ and $\tau>0$. Inequality~\eqref{eq:liminf_tau_eps} thus follows. Therefore, passing to the limit in~\eqref{eq:liminf_tau_eps} first as~$\tau\to0$ and then as~$\eps\to0$, we get~\eqref{eq:theorem_weak_main}. This concludes the proof.

\section{Proof of Theorem~\ref{th:BV_vs_S}}
\label{sec:BV_vs_S}

In this section, we prove \cref{th:BV_vs_S}. The idea of the proof is to estimate the quantity~$\A f$ defined in~\eqref{eq:def_A} from below with the integral average of~$|Df|$, in the spirit of the reverse Poincaré's inequality. Obviously, it is not possible to get such an estimate for arbitrary $BV$~functions (with a constant that does not depend on the function itself). 

However, it is simple to see that a reverse Poincaré's inequality is true for one-variable monotone functions, so that one would expect that a sort of reverse Poincaré's inequality may hold for arbitrary $BV$~functions if first one specifies a direction $\nu\in\R^d$, $|\nu|=1$, and then adds a suitable correction term measuring how far is~$f$ from being dependent only on the direction~$\nu$ and monotone.

\begin{lemma}\label{lemma:penalized_reverse_P}
There exist two dimensional constants $C_1,C_2>0$ such that
\begin{equation}\label{eq:penalized_reverse_P}
\frac{1}{r}\int_{B(x,C_2r)}\abs*{f-(f)_{x,C_2r}}\ dy
+\int_{B(x,C_2r)}\left( 1-\scalar*{\nu,\eta}\right)\ d|D f|
\ge
C_1\int_{B(x,r)} d|D f|
\end{equation}
for all $\nu\in\R^d$ with $|\nu|=1$, $x\in\R^d$, $r>0$ and $f\in BV(\R^d)$, where $\eta\colon\R^d\to\R^d$ satisfies $D f=\eta\,|D f|$ and $|\eta|=1$ $|D f|$-a.e.\ in~$\R^d$.
\end{lemma}

\begin{proof}
We claim that there exist two dimensional constants $C_1,C_2>0$ such that
\begin{equation}\label{eq:mary_smooth}
\frac{1}{r}\int_{B(x,C_2r)}\abs*{f-(f)_{x,C_2r}}\ dy+\int_{B(x,C_2r)}\left( |\nabla f|-\scalar*{\nu,\nabla f}\right)\ dy
\ge
C_1\int_{B(x,r)}\abs*{\nabla f}\ dy
\end{equation}
for all $\nu\in\R^d$ with $|\nu|=1$, $x\in\R^d$, $r>0$ and $f\in C^\infty(\R^d)$. By a standard rescaling argument, we just need to prove that there exists a dimensional constant $C>0$ such that
\begin{equation}\label{eq:mary_Q}
\int_{2Q}\abs*{f-(f)_{2Q}}\ dx
+\int_{2Q} \left(|\nabla f|-\scalar*{\mathrm{e}_1,\nabla f}\right)\ dy
\ge
C\int_{Q}\abs*{\nabla f}\ dy
\end{equation}
for all $f\in C^\infty(\R^d)$, where $Q=[-1,1]^d$ and $(f)_{2Q}:=\dashint_{2Q}f\,dy$. We prove~\eqref{eq:mary_Q} by contradiction. For all $n\in\N$, assume there exists $f_n\in C^\infty(\R^d)$ such that
\begin{equation*}
\int_Q |\nabla f_n|\ dy=1,
\quad 
(f_n)_{2Q}=0,
\quad 
\int_{2Q}\abs*{f_n}\ dy
+\int_{2Q} \left(|\nabla f_n|-\scalar*{\mathrm{e}_1,\nabla f_n}\right)\ dy
<\frac{1}{n}.
\end{equation*}
Then consider $g_n\colon 2I\to\R$, $I=[-1,1]$, defined as 
\begin{equation*}
g_n(t):=\int_{[-1,1]^{d-1}} g_n(t,y')\ dy',
\qquad
t\in 2I.
\end{equation*}
For all $s\in[0,1]$, we have
\begin{equation*}
\begin{split}
\frac{1}{n}
&\ge\int_{2Q}\left(|\nabla f_n|-\scalar*{\mathrm{e}_1,\nabla f_n}\right)\ dy\\
&\ge\int_{-1-s}^{1+s}\int_{[-1,1]^{d-1}}|\nabla f_n(y_1,y')|-\scalar*{\mathrm{e}_1,\nabla f_n(y_1,y')}\ dy'\,dy_1\\
&\ge 1-\int_{-1-s}^{1+s}g_n'(y_1)\ dy_1
=1-\big(g_n(1+s)-g_n(-1-s)\big),
\end{split}
\end{equation*}
so that
\begin{equation*}
g_n(1+s)-g_n(-1-s)\ge 1-\frac{1}{n},
\qquad
n\in\N.
\end{equation*}
This contradicts the fact that $g_n\to 0$ in $L^1(2I)$, since $\|g_n\|_{L^1(2I)}\le\|f_n\|_{L^1(2Q)}\le\frac{1}{n}$. This concludes the proof of~\eqref{eq:mary_Q} and thus inequality~\eqref{eq:mary_smooth} follows.

We can now conclude the proof by a standard approximation argument. Given $f\in BV(\R^d)$, by~\cite{AFP00}*{Theorem~3.9 and Proposition~3.13} we can find $f_n\in BV(\R^d)\cap C^\infty(\R^d)$ such that $f_n\to f$ in $L^1(\R^d)$, $\|\nabla f_n\|_{L^1(\R^d)}\to|D f|(\R^d)$ and $|\nabla f_n|\,\leb^d\weakstarto|D f|$ in $\R^d$ as $n\to+\infty$. Therefore, inequality~\eqref{eq:penalized_reverse_P} follows by Reshetnyak's continuity Theorem, see~\cite{AFP00}*{Theorem~2.39}.
\end{proof}

\begin{remark}
We must have $C_2>1$ in \cref{lemma:penalized_reverse_P}, as the following example shows. For $n\in\N$, consider $f_n\colon I\to\R$, $I=[-1,1]$, defined as
\begin{equation*}
f_n(x)=
\begin{cases}
	nx+n-1 &\quad -1\le x<-\left(1-\frac{1}{n}\right)\\
	0                 &\quad -\left(1-\frac{1}{n}\right)\le x\le 1-\frac{1}{n}\\
	nx+1-n  &\quad 1-\frac{1}{n}<x\le 1.
\end{cases}
\end{equation*}
Then $(f_n)=0$, $\|f_n\|_{L^1(I)}=\frac{1}{n}$ and $\|f_n'\|_{L^1(I)}=2$, so that inequality~\eqref{eq:penalized_reverse_P_BV} with $C_2=1$, $v=1$, $x=0$ and $r=1$ would imply $C_1\ge 2n$, a contradiction.
\end{remark}

We will not apply \cref{lemma:penalized_reverse_P} directly, but we will use the following easy consequence of it. There exist two dimensional constants $C_1,C_2>0$ such that
\begin{equation}\label{eq:penalized_reverse_P_any_norm}
\frac{1}{r}\int_{B(x,C_2r)}\abs*{f-(f)_{x,C_2r}}\ dy
+2\int_{B(x,C_2r)} |\eta-v|\ d|D f|
\ge
C_1\int_{B(x,r)} d|D f|
\end{equation}
for all $v\in\R^d$, $x\in\R^d$, $r>0$ and $f\in BV(\R^d)$, where $\eta\colon\R^d\to\R^d$ is as in \cref{lemma:penalized_reverse_P}. 
The proof of~\eqref{eq:penalized_reverse_P_any_norm} is immediate. Indeed, since we can assume $v\ne0$ without loss of generality, we just need to notice that
\begin{equation*}
2|\eta-v|\ge 1-\scalar*{\frac{v}{|v|},\eta}
\qquad
\text{$|D f|$-a.e.\ in~$\R^d$},
\end{equation*}
and apply \cref{lemma:penalized_reverse_P} with $\nu=v/|v|$.

Having inequality~\eqref{eq:penalized_reverse_P_any_norm} at our disposal, we are now ready to prove \cref{th:BV_vs_S}.

\begin{proof}[Proof of \cref{th:BV_vs_S}]
Fix $f\in BV(\R^d)$ and write $D f=\eta\,|D f|$, where $\eta\colon\R^d\to\R^d$ is as in \cref{lemma:penalized_reverse_P}. For each $\eps>0$, let $\eta_\eps\in\Lip(\R^d)$ be a Lipschitz function with Lipschitz constant $C_\eps>0$ such that the measure $\nu_\eps:=|\eta-\eta_\eps|\,|D f|$ satisfies
\begin{equation}\label{eq:def_eta_eps_BV}
\nu_\eps(\R^d)=\int_{\R^d} |\eta-\eta_\eps|\ d|D f| < \eps.
\end{equation}

Now fix $x\in\R^d$. Applying \cref{lemma:penalized_reverse_P} with $v=\eta_\eps(x)$, we get
\begin{equation}\label{eq:penalized_reverse_P_BV}
\frac{1}{r}\int_{B(x,C_2r)}\abs*{f(y)-(f)_{x,C_2r}}\ dy
+2\int_{B(x,C_2r)} |\eta(y)-\eta_\eps(x)|\ d|D f|(y)
\ge 
C_1\int_{B(x,r)} d|D f|(y).
\end{equation}
Note that
\begin{equation}\label{eq:error_reverse_BV}
\begin{split}
\int_{B(x,C_2r)} & |\eta(y)-\eta_\eps(x)|\ d|D f|(y)\\
&\le\int_{B(x,C_2r)} |\eta(y)-\eta_\eps(y)|\ d|D f|(y)
+\int_{B(x,C_2r)} |\eta_\eps(y)-\eta_\eps(x)|\ d|D f|(y)\\
&\le\int_{B(x,C_2r)} d\nu_\eps(y)+C_2C_\eps r\int_{B(x,C_2r)} d|D f|(y).
\end{split}
\end{equation}
Combining~\eqref{eq:penalized_reverse_P_BV} and~\eqref{eq:error_reverse_BV}, we get that
\begin{equation}\label{eq:ineq1_BV}
\A f(x)+\M\nu_\eps(x)+\frac{C_\eps r}{\omega_d(C_2r)^d}\int_{B(x,C_2r)} d|D f|
\gtrsim_d
\frac{1}{\omega_dr^d}\int_{B(x,r)} d|D f|,
\end{equation}
for all $x\in\R^d$ and $r>0$, where~$\A f$ is the function defined in~\eqref{eq:def_A}.

Now fix $\tau>0$. Taking the supremum for $r\in(0,\tau)$ in~\eqref{eq:ineq1_BV} and recalling the definition in~\eqref{eq:def_M_stopped}, we get that
\begin{equation*}
\A f(x)+\M\nu_\eps(x)+C_\eps\tau\,\M^{C_2\tau}(D f)(x)
\gtrsim_d
\M^\tau(D f)(x).
\end{equation*}
Thus, by the observation made in~\eqref{eq:magic_stop}, inequalities~\eqref{eq:def_eta_eps_BV}, \eqref{eq:standard_limsup} and \cref{prop:main}, we conclude that
\begin{equation*}
\begin{split}
|D^s f|(\R^d)&\lesssim_d\liminf_{\lambda\to+\infty}\lambda\,\leb^d(\set*{x\in\R^d : \M (D f)(x)>\lambda})\\
&=\liminf_{\lambda\to+\infty}\lambda\,\leb^d(\set*{x\in\R^d : \M^\tau(D f)(x)>\lambda})\\
&\lesssim_d\liminf_{\lambda\to+\infty}\lambda\,\leb^d(\set*{x\in\R^d : \A f(x)>\lambda})\\
&\quad+\limsup_{\lambda\to+\infty}\lambda\,\leb^d(\set*{x\in\R^d : \M\nu_\eps(x)>\lambda})\\
&\quad+C_\eps\tau\limsup_{\lambda\to+\infty}\lambda\,\leb^d(\set*{x\in\R^d : \M ^{C_2\tau}(D f)(x)>\lambda})\\
&\lesssim_d\liminf_{\lambda\to+\infty}\lambda\,\leb^d(\set*{x\in\R^d : \A f(x)>\lambda})
+\nu_\eps(\R^d)+C_\eps\tau|D^s f|(\R^d)\\
&\lesssim_d\liminf_{\lambda\to+\infty}\lambda\,\leb^d(\set*{x\in\R^d : \A f(x)>\lambda})
+\eps+C_\eps\tau|D^s f|(\R^d),\\
\end{split}
\end{equation*}
for all $\eps>0$ and $\tau>0$. Passing to the limit first as~$\tau\to0$ and then as~$\eps\to0$, we get
\begin{equation*}
\liminf_{\lambda\to+\infty}\lambda\,\leb^d(\set*{x\in\R^d : \A f(x)>\lambda})\gtrsim_d
|D^s f|(\R^d).
\end{equation*}
This concludes the proof.
\end{proof}

\section{Proof of \cref{th:decay}}
\label{sec:proof_decay}

In this section, we prove \cref{th:decay}. Let $b\colon(0,T)\times\R^d\to\R^d$ be a vector field satisfying~\eqref{eq:assumptions_B}. By~\cite{N18}*{Remark 10}, the quantity
\begin{equation*}
\mathcal{Q}(B;\delta)=|\log\delta|^{-1}
\int_{0}^{T}\int_B\min\set*{\delta^{-1},\M|Db_t|(x)}\ dx\,dt
\end{equation*}
defined in~\eqref{eq:decay_quantity} satisfies
\begin{equation*}
\limsup_{\delta\to0^+}\mathcal{Q}(B;\delta)
\lesssim_d\int_{0}^{T} |D^sb_t|(\closure{B})\ dt
\end{equation*}
for all balls $B\subset\R^d$. This proves the second part of~\eqref{eq:th_decay}. To prove the first part of~\eqref{eq:th_decay}, fix a ball $B=B_r\subset\R^d$ of radius $r>0$. We claim that  
\begin{equation}\label{eq:decay_claim}
\liminf_{\delta\to0^+}|\log\delta|^{-1}
\int_B\min\set*{\delta^{-1},\M|Db_t|(x)}\ dx
\gtrsim_d
|D^sb_t|(B_r)
\end{equation}
holds for a.e.\ $t\in(0,T)$, so that the conclusion follows by Fatou's Lemma. Indeed, for any $\eps\in(0,r/2)$ and for a.e.\ $t\in(0,T)$, we can estimate 
\begin{equation*}
\int_{B_r}\min\set*{\delta^{-1},\M|Db_t|(x)}\ dx\ge
\int_{\delta^{-1/2}}^{\delta^{-1}}\leb^d\left(\set*{x\in B_r : \M\big(\mathbf{1}_{B_{r-\eps}}|Db_t|\big)(x)>\lambda}\right)d\lambda.
\end{equation*}
Note that $\M\big(\mathbf{1}_{B_{r-\eps}}|Db_t|\big)(x)\lesssim_d |Db_t|(B_{r})\varepsilon^{-d}$ for every $x\in\R^d\setminus B_r$ and for a.e.\ $t\in(0,T)$. Indeed, if $|x|>r$ then $B_{r-\eps}\cap B(x,s)=\varnothing$ for all $s\in[0,\eps]$, so that
\begin{equation*}
\M\big(\mathbf{1}_{B_{r-\eps}}|Db_t|\big)(x)
=\sup_{s>0}\frac{|Db_t|(B_{r-\eps}\cap B(x,s))}{\leb^d(B(x,s))}
\le\frac{|Db_t|(B_{r-\eps})}{\leb^d(B(x,\eps))}
\lesssim_d |Db_t|(B_r)\eps^{-d}.
\end{equation*}
Now, using the decomposition
\begin{align*}
\leb^d\left(\set*{x\in B_r : \M\big(\mathbf{1}_{B_{r-\eps}}|Db_t|\big)(x)>\lambda}\right)
=&\leb^d\left(\set*{x\in \R^d : \M\big(\mathbf{1}_{B_{r-\eps}}|Db_t|\big)(x)>\lambda}\right)\\
&-\leb^d\left(\set*{x\in \R^d\setminus B_r : \M\big(\mathbf{1}_{B_{r-\eps}}|Db_t|\big)(x)>\lambda}\right)
\end{align*}
and \cref{prop:main}, we get that 
\begin{equation*}
\inf_{\lambda\in (\delta^{-1/2},\delta^{-1})}\lambda\,\leb^d\left(\set*{x\in B_r : \M\big(\mathbf{1}_{B_{r-\eps}}|Db_t|\big)(x)>\lambda}\right)
\gtrsim_d
|D^sb_t|(B_{r-\eps})
\end{equation*}
for all $\delta\lesssim_d|Db_t|(B_{r})^{-1/2}\varepsilon^{d/2}$ and for a.e.\ $t\in (0,T)$. Hence, for $\delta>0$ sufficiently small, we obtain that
\begin{equation*}
|\log\delta|^{-1}\int_{B_r}\min\set*{\delta^{-1},\M|Db_t|(x)}\ dx
\gtrsim_d |\log\delta|^{-1}\int_{\delta^{-1/2}}^{\delta^{-1}}|D^sb_t|(B_{r-\eps})\ \frac{d\lambda}{\lambda}
\gtrsim_d |D^sb_t|(B_{r-\eps}).
\end{equation*}
Therefore 
\begin{equation*}
\liminf_{\delta\to0^+}|\log\delta|^{-1}
\int_B\min\set*{\delta^{-1},\M|Db_t|(x)}\ dx
\gtrsim_d
|D^sb_t|(B_{r-\varepsilon}),
\end{equation*}
so that claim~\eqref{eq:decay_claim} follows by letting $\eps\to0^+$. 

\appendix

\section{Proof of \cref{prop:main}}
\label{sec:proof_prop_main}

In this section, we prove \cref{prop:main}. The main ingredient of the argument is the following well-known reverse weak $(1,1)$-type inequality for the maximal function in~\eqref{eq:def_max_func_f}. For the proof, which uses Calderon--Zygmund decomposition, we refer to~\cite{S70}*{Chapter~1, Section~5}.

\begin{lemma}\label{lemma:CZ}
There exists a dimensional constant $C>0$ such that
\begin{equation}\label{eq:CZ}
t\,\leb^d(\set*{x\in\R^d : Mf(x)>C t})\gtrsim_d \int_{\set{f>t}} f\ d\leb^d
\end{equation}	
for all $t>0$ and all non-negative $f\in L^1(\R^d)$. More in general, there exists a dimensional constant $C>0$ such that
\begin{equation}\label{eq:CZ_local}
t\,\leb^d(\set*{x\in Q : Mf(x)>C t})\gtrsim_d \int_{\set{f>t}\cap Q} f\ d\leb^d
\end{equation}
for all cube $Q\subset\R^d$, for all $t>\dashint_Q f\ d\leb^d$ and all non-negative $f\in L^1_{\rm loc}(\R^d)$.
\end{lemma}

We are now ready to prove \cref{prop:main}.

\begin{proof}[Proof of \cref{prop:main}]
Without loss of generality, we can assume that $\mu$ is non-negative and singular with respect to~$\leb^d$.

For all $\eps>0$ define $f_\eps\in L^1(\R^d)$ as
\begin{equation*}
f_{\eps}(x):=\frac{\mu(B(x,\eps))}{\leb^d(B(x,\eps))}=\frac{\mu(B(x,\eps))}{\omega_d\,\eps^d},
\qquad
x\in\R^d,
\end{equation*}
where $\omega_d=\leb^d(B(0,1))$.

We claim that $(f_\eps)_{\eps>0}$ satisfies the following \emph{almost semigroup property}: for all $r>0$ and $x\in\R^d$, it holds
\begin{equation}\label{eq:semigroup}
(f_\eps)_{x,r}:=\dashint_{B(x,r)} f_\eps\ dy
\lesssim_d
f_{r+\eps}(x).
\end{equation}
Indeed, fix $x\in\R^d$ and $r>0$. By Tonelli's Theorem, we can write
\begin{equation*}
(f_{\eps})_{x,r}
=\frac{1}{\omega_d \eps^d}\frac{1}{\omega_d r^d}\int_{\R^d}\leb^d\big(B(x,r)\cap B(y,\eps)\big)\ d\mu(y).
\end{equation*}
Since
\begin{equation*}
\leb^d\big(B(x,r)\cap B(y,\eps)\big)\leq\mathbf{1}_{B(x,r+\eps)}(y)\,\min\set{\omega_d\eps^d,\omega_d r^d},
\end{equation*}
for all $y\in\R^d$, we deduce that
\begin{equation*}
(f_\eps)_{x,r}
\le\frac{\mu(B(x,r+\eps))}{\omega_d(r+\eps)^d}\,
\frac{\omega_d(r+\eps)^d\min\set{\omega_d\eps^d,\omega_d r^d}}{\omega_d \eps^d\ \omega_dr^d}
\le 2^d f_{r+\eps}(x).
\end{equation*}
This concludes the proof of~\eqref{eq:semigroup}.

Thanks to~\eqref{eq:semigroup}, we easily get 
\begin{equation*}
\M f_\eps(x)\lesssim_d \M\mu(x),
\end{equation*} 
for all $x\in\R^d$. Thus, by \cref{lemma:CZ}, we conclude that
\begin{equation}\label{eq:eps_ineq}
t\leb^d(\set*{x\in\R^d : \M\mu(x)>C_dt})
\gtrsim_d 
\int_{\set{f_\eps>t}} f_\eps\ dx,
\end{equation}
for all $t>0$ and all $\eps>0$, where $C_d>0$ is a dimensional constant. 

We now claim that
\begin{equation}\label{eq:limsup_eps_ineq}
\limsup_{\eps\to0^+} \int_{\set{f_\eps>t}} f_\eps\ dx
\gtrsim_d 
\mu(\R^d),
\end{equation}
for all $t>0$, so that~\eqref{eq:main} follows immediately combining~\eqref{eq:eps_ineq} and~\eqref{eq:limsup_eps_ineq}. Indeed, by Tonelli's Theorem we have
\begin{equation}
\int_{\set{f_\eps>t}} f_\eps\ dx
=\int_{\R^d}\frac{\leb^d(\set{f_\eps>t}\cap B(x,\eps))}{\omega_d \eps^d}\ d\mu(x),
\end{equation}
for all $\eps>0$. Hence, by Fatou's Lemma, we get that 
\begin{equation*}
\limsup_{\eps\to0^+} \int_{\set{f_\eps>t}} f_\eps\ dx
\ge
\int_{\R^d}\liminf_{\eps\to0^+}\frac{\leb^d(\set{f_\eps>t}\cap B(x,\eps))}{\omega_d \eps^d}\ d\mu(x).
\end{equation*}

We now claim that
\begin{equation}\label{eq:limsup_eps_ineq_2}
\leb^d(\set{f_\eps>t}\cap B(x,\eps))
\ge 
\frac{1}{2^d}\omega_d \eps^d,
\end{equation}
for $\mu$-a.e.\ $x\in\R^d$ and all $\eps>0$. To prove~\eqref{eq:limsup_eps_ineq_2}, we need to observe two preliminary facts. 

First, notice that, given $\eps>0$ and $t>0$, we have
\begin{equation}\label{eq:inclusion}
f_{\eps/2}(x)>2^d t
\implies
B(x,\eps/2)
\subset
\set{x\in\R^d : f_\eps(x)>t}.
\end{equation}
Implication~\eqref{eq:inclusion} follows from the trivial inclusion $B(x,\eps/2)\subset B(y,\eps)$ for all $y\in B(x,\eps/2)$. 

Second, notice that
\begin{equation}\label{eq:explosion}
\lim_{\eps\to0^+} f_\eps(x)=+\infty,
\end{equation}
for $\mu$-a.e.\ $x\in\R^d$. Indeed, we have
\begin{equation}\label{eq:covering_explosion}
\set*{x\in\R^d : \liminf_{\eps\to0^+} f_\eps(x)<+\infty}\subset\bigcup_{n\in\N} A_n,
\end{equation}
where
\begin{equation*}
A_n:=\set*{x\in\R^d : \liminf_{\eps\to0^+}\frac{\mu (B(x,\eps))}{\omega_d\eps^d}\leq n}.
\end{equation*}
By a standard covering argument (for instance apply Vitali's covering Lemma, see~\cite{S70}*{Section~1.6}), one can prove that 
\begin{equation*}
\mu(E)\leq n\leb^d(E)
\quad
\text{for all Borel sets}\ E\subset A_n.
\end{equation*}
Since $\mu$ is singular with respect to~$\leb^d$, we must have that $\mu(A_n)=0$ for all $n\in\N$ and thus, by~\eqref{eq:covering_explosion}, we conclude that
\begin{equation*}
\mu\left(\set*{x\in\R^d : \liminf_{\eps\to0^+} f_\eps(x)<+\infty}\right)=0.
\end{equation*}

We can now prove~\eqref{eq:limsup_eps_ineq_2}. Fix $x\in\R^d$ such that~\eqref{eq:inclusion} holds true. Then there exists $\eps_x>0$ such that $f_{\eps/2}(x)>2^d t$ for all $\eps<\eps_x$. Hence $B(x,\eps/2)\subset \set{f_\eps>t}$ and so
\begin{equation*}
\leb^d(\set{\mu_{\eps}>t}\cap B(x,\eps))
\ge
\frac{1}{2^d}\omega_d \eps^d
\end{equation*}
for all $\eps<\eps_x$. Thus~\eqref{eq:limsup_eps_ineq_2} follows and the proof of~\eqref{eq:main} is complete. 

The proof of the local inequality~\eqref{eq:main_local} similarly follows from~\eqref{eq:CZ_local} and is left to the reader.
\end{proof}



\begin{bibdiv}
\begin{biblist}

\bib{A04}{article}{
   author={Ambrosio, Luigi},
   title={Transport equation and Cauchy problem for $BV$ vector fields},
   journal={Invent. Math.},
   volume={158},
   date={2004},
   number={2},
   pages={227--260},
}

\bib{AFP00}{book}{
   author={Ambrosio, Luigi},
   author={Fusco, Nicola},
   author={Pallara, Diego},
   title={Functions of bounded variation and free discontinuity problems},
   series={Oxford Mathematical Monographs},
   publisher={The Clarendon Press, Oxford University Press, New York},
   date={2000},
}

\bib{CDeL08}{article}{
   author={Crippa, Gianluca},
   author={De Lellis, Camillo},
   title={Estimates and regularity results for the DiPerna-Lions flow},
   journal={J. Reine Angew. Math.},
   volume={616},
   date={2008},
   pages={15--46},
}

\bib{DL89}{article}{
   author={DiPerna, R. J.},
   author={Lions, P.-L.},
   title={Ordinary differential equations, transport theory and Sobolev spaces},
   journal={Invent. Math.},
   volume={98},
   date={1989},
   number={3},
   pages={511--547},
}

\bib{E10}{book}{
   author={Evans, Lawrence C.},
   title={Partial differential equations},
   series={Graduate Studies in Mathematics},
   volume={19},
   edition={2},
   publisher={American Mathematical Society, Providence, RI},
   date={2010},
   pages={xxii+749},
}

\bib{MS13}{book}{
   author={Muscalu, Camil},
   author={Schlag, Wilhelm},
   title={Classical and multilinear harmonic analysis. Vol. I},
   series={Cambridge Studies in Advanced Mathematics},
   volume={137},
   publisher={Cambridge University Press, Cambridge},
   date={2013},
}

\bib{N18}{article}{
   author={Nguyen, Quoc Hung},
   title={Quantitative estimates for regular Lagrangian flows with $BV$ vector fields},
   status={preprint},
   eprint={https://arxiv.org/abs/1805.01182},
   date={2018}
}

\bib{J10}{article}{
  author={Jabin, Pierre-Emmanuel},
  title={Differential equations with singular fields},
  journal={J. Math. Pures Appl. (9)},
  volume= {94},
  year={2010},
  number={6},
  pages={597--621},
}

\bib{S70}{book}{
   author={Stein, Elias M.},
   title={Singular integrals and differentiability properties of functions},
   series={Princeton Mathematical Series, No. 30},
   publisher={Princeton University Press, Princeton, N.J.},
   date={1970},
   pages={xiv+290},
}

\end{biblist}
\end{bibdiv}
\end{document}